\documentclass[12pt, twoside, leqno]{article}
\usepackage{tikz}
\usepackage{amssymb}
\usepackage{amsmath}
\usepackage{amsthm}
\usepackage{latexsym}
\usepackage{graphicx}
\usepackage{verbatim}
\usepackage{xcolor}
\usepackage{enumitem}
\usepackage{float}
\usepackage{mathtools}
\usepackage{authblk}

\newtheorem{theorem}{Theorem}[section]

\newtheorem{lemma}[theorem]{Lemma}
\newtheorem{proposition}[theorem]{Proposition}

\theoremstyle{definition}
\newtheorem{definition}[theorem]{Definition}

\newtheorem{example}[theorem]{Example}
\newtheorem{question}[theorem]{Question}

\numberwithin{equation}{section}

\usetikzlibrary{patterns, arrows.meta}
\usetikzlibrary{decorations.markings}

\newcommand{\Pol}{\mathcal{P}}

\newcommand{\abs}[1]{\left\lvert#1\right\rvert}
\newcommand{\var}{\mathop{\mathrm{var}}}
\newcommand{\cvar}{\mathop{\mathrm{cvar}}}

\newcommand{\vf}{\mathop{\mathrm{vf}}}
\newcommand{\BV}{\mathop{{BV}}}

\newcommand{\Borel}{\mathop{\mathrm{Borel}}}
\newcommand{\ds}{\displaystyle}
\newcommand{\norm}[1]{\left \Vert #1 \right \Vert}

\newcommand{\st}{\,:\,}

\newcommand{\C}{\mathbb{C}}
\newcommand{\R}{\mathbb{R}}

\def\ls[#1,#2]{\overline{\vphantom{\vbox to 1.2 ex{}} #1\, #2}}
\def\seg(#1,#2){\ls[\vecx_{#1},\vecx_{#2}]}
\newcommand{\veca}{{\boldsymbol{a}}}
\newcommand{\vecb}{{\boldsymbol{b}}}
\newcommand{\vecc}{{\boldsymbol{c}}}

\newcommand{\vecx}{{\boldsymbol{x}}}
\newcommand{\vecy}{{\boldsymbol{y}}}

\newcommand{\vecv}{{\boldsymbol{v}}}

\newcommand{\vecz}{{\boldsymbol{z}}}

\newcommand{\vecd}{{\boldsymbol{d}}}

\newcommand{\vecalpha}{{\boldsymbol{\alpha}}}
\newcommand{\vecbeta}{{\boldsymbol{\beta}}}
\newcommand{\vecgamma}{{\boldsymbol{\gamma}}}
\newcommand{\vecdelta}{{\boldsymbol{\delta}}}
\numberwithin{equation}{section}
\renewcommand{\Re}{\mathop{\mathrm{Re}}}

\newcommand{\mC}{\mathbb{C}}
\newcommand{\mT}{\mathbb{T}}

\newcommand{\Pt}{{\mathcal P}}

\def\implies{\Rightarrow}
\def\inpr#1,#2{\t \hbox{\langle #1 , #2 \rangle} \t}
\def\ip<#1,#2>{\langle #1,#2 \rangle}

\def\norm#1{\left \Vert #1 \right \Vert}
\def\paren(#1){\left( #1 \right)}

\def\sparen(#1){\Bigl ( #1 \Bigr )}
\def\ip<#1,#2>{\langle #1, #2 \rangle}

\newcommand{\calA}{\mathcal{A}}
\newcommand{\calB}{\mathcal{B}}

\newcommand{\calC}{\mathcal{C}}

\renewcommand{\Re}{\mathop{\mathrm{Re}}}
\renewcommand{\Im}{\mathop{\mathrm{Im}}}

\newcommand{\normbv}[1]{\left\lVert#1\right\rVert_{\BV(\sigma)}}

\newenvironment{acknowledgements}{\medskip
\noindent
\textit{Acknowledgements.}\ }{\medskip}

\begin{document}

\title{Isomorphisms of $BV(\sigma)$ spaces}

\author{Shaymaa Al-shakarchi${}^{a}$}
\author{Ian Doust}

\affil{School of Mathematics and Statistics, 
University of New South Wales, 
UNSW Sydney 2052, Australia, 
Email: i.doust@unsw.edu.au }

\affil{${}^{a}$ Current address: Department of Mathematics, Faculty of Basic Education, University of Kufa, Najaf, Iraq}

\date{}

\maketitle

\renewcommand{\thefootnote}{}



\renewcommand{\thefootnote}{\arabic{footnote}}
\setcounter{footnote}{0}

\begin{abstract}
In this paper we investigate the relationship between the properties of a compact set $\sigma \subseteq \C$ and the structure of the space $BV(\sigma)$ of functions of bounded variation (in the sense of Ashton and Doust) defined on $\sigma$. For the subalgebras of absolutely continuous functions on $\sigma$, it is known that for certain classes of compact sets one obtains a Gelfand--Kolmogorov type result: the function spaces $AC(\sigma_1)$ and $AC(\sigma_2)$ are isomorphic if and only if the domain sets $\sigma_1$ and $\sigma_2$ are homeomorphic. Our main theorem is that in this case the isomorphism must extend to an isomorphism of the $BV(\sigma)$ spaces. An application is given to the spectral theory of $AC(\sigma)$ operators.
\end{abstract}

\section{Introduction}

Many of the central results in the spectral theory of linear operators involve the extension of a homomorphism $\Phi: \calA \to \calC$ between Banach algebras to a larger domain $\calB \supseteq \calA$. Examples include the spectral theorem for normal operators on a Hilbert space,  where one extends a $C(\sigma(T))$ functional calculus for an operator $T$ to a functional calculus for the bounded Borel measurable functions $\Borel(\sigma(T))$, or the spectral theorem for well-bounded operators on a reflexive Banach space, where one extends an $AC[a,b]$ functional calculus to $BV[a,b]$, the functions of bounded variation on $[a,b]$.

In some situations, special properties of the Banach algebras can be used to provide easy extension results.
The classical Gelfand--Kolmogorov Theorem says that two compact Hausdorff spaces $K_1$ and $K_2$ are homeomorphic if and only if  the  Banach algebras $C(K_1)$ and $C(K_2)$ are isomorphic (as Banach algebras). Every isomorphism $\Phi: C(K_1) \to C(K_2)$ is of the form $\Phi(f) = f \circ h^{-1}$ for some homeomorphism $h: K_1 \to K_2$. Conversely, every such homeomorphism clearly generates an algebra isomorphism $\Phi$, and this map obviously extends  to a Banach algebra isomorphism from $\Borel(K_1)$ to $\Borel(K_2)$.

In order to provide a general theory which includes both well-bounded and trigonometrically well-bounded operators, Ashton and Doust \cite{AD1} introduced two new families of Banach algebras of functions. Given a nonempty compact subset $\sigma$ of the plane, $BV(\sigma)$ denotes the set of functions $f: \sigma \to \mC$ of bounded variation on $\sigma$, and $AC(\sigma)$ denotes the subalgebra of absolutely continuous functions. (Full definitions are given in Section~\ref{S:Prelim}.) A bounded operator on a Banach space $X$ which admits an $AC(\sigma)$ functional calculus is called an $AC(\sigma)$ operator. The properties of these operators were studied in \cite{AD2}. At least on reflexive spaces, the $AC(\sigma)$ spaces play a corresponding role in the spectral theory of $AC(\sigma)$ operators to that played by $C(\sigma)$ spaces in the theory of normal operators on Hilbert space.

A natural question was to determine when two spaces $AC(\sigma_1)$ and $AC(\sigma_2)$ are isomorphic as Banach algebras.
Doust and Leinert \cite{DL1} showed that one gets one direction of a Gelfand--Kolmogorov type theorem in the setting of these spaces. If $\Phi: AC(\sigma_1) \to AC(\sigma_2)$ is an algebra isomorphism, then there exists a homeomorphism $h: \sigma_1 \to \sigma_2$ such that $\Phi(f) = f \circ h^{-1}$. In general the converse is false.  That is, not every homeomorphism $h: \sigma_1 \to \sigma_2$ generates an algebra isomorphism of the associated $AC(\sigma)$ spaces.  One can however obtain positive results if one restricts the class of compact sets that are considered. For example, if $\sigma_1$ and $\sigma_2$ are any two polygons (which we take here to include their interiors) then $AC(\sigma_1)$ and $AC(\sigma_2)$ are isomorphic \cite[Theorem~6.3]{DL1}. Similar theorems for other classes of subsets of the plane can be found in \cite{DAS} and \cite{ASD}.

In each of these papers, an intermediate step in showing that two $AC(\sigma)$ spaces were indeed isomorphic was to show that the corresponding $BV(\sigma)$ spaces were in fact isomorphic. One then needed to show that the isomorphism preserves the subalgebras of absolutely continuous functions. The aim of this paper is to show that every isomorphism of $AC(\sigma)$ spaces must be the restriction of an isomorphism of $BV(\sigma)$ spaces. Said another way, every Banach algebra isomorphism $\Phi: AC(\sigma_1) \to AC(\sigma_2)$ extends to an isomorphism from $BV(\sigma_1)$ to $BV(\sigma_2)$.

It should be noted that even in the case that a homeomorphism $h: \sigma_1 \to \sigma_2$ does generate an isomorphism $\Phi_h$ of the associated $BV(\sigma)$ spaces, this map need not preserve the $AC(\sigma)$ spaces. Examples illustrating some of the possible behaviour are given is Sections~\ref{S:Isom} and \ref{S:ACBV}.

In the final section we give an application which shows that for a class of sets $\sigma$, the functional calculus for every $AC(\sigma)$ operator on a reflexive Banach space can be extended to all the functions of bounded variation.

\section{Preliminaries}\label{S:Prelim}
In this section we shall briefly outline the definitions of the spaces $AC(\sigma)$ and $BV(\sigma)$. We shall follow a simplified development rather than the original one given in \cite{AD1}. Further details about the evolution of this definition of variation for functions defined on subsets of the plane is given in the appendix.

For the remainder of the paper, unless otherwise specified, isomorphism will mean a Banach algebra isomorphism, that is, a continuous algebra isomorphism  with a continuous inverse. It is worth noting that the $AC(\sigma)$ and $BV(\sigma)$ spaces are always semisimple Banach algebras. A consequence of this is that any algebra isomorphism between say $BV(\sigma)$ spaces is automatically a Banach algebra isomorphism. We shall write $\mathcal{A} \simeq \mathcal{B}$ to denote that $\mathcal{A}$ is isomorphic to $\mathcal{B}$.  All algebras will consist of complex-valued functions. We shall identify the plane as either $\C$ or $\R^2$ as is notationally convenient. Given two distinct points $\vecx$ and $\vecy$ in the plane, $\ls[\vecx,\vecy]$ will denote the closed line segment joining them.

Suppose that $\sigma$ is a nonempty compact subset of $\mC$ and that $f: \sigma \to \mathbb{C}$. The partitions which are used in the classical definition of variation are replaced here by finite lists of points in the domain $\sigma$. The definition needs to take into account not just the differences between the function values at points on the list, but also how these points are positioned in the plane.

Let $S = [\vecx_0,\vecx_1,\dots,\vecx_n]$ be a finite ordered list of elements of $\sigma$, where, for the moment, we shall assume that $n \ge 1$.
Let $\gamma_S$ denote the piecewise linear curve joining the points of $S$ in order. We shall usually assume that no two consecutive points are equal, but otherwise we do not require that the elements of such a list are distinct.

The \textit{curve variation of $f$ on the ordered set $S$} is defined to be
\begin{equation*} \label{lbl:298}
    \cvar(f, S) =  \sum_{i=1}^{n} \abs{f(\vecx_{i}) - f(\vecx_{i-1})}.
\end{equation*}
Associated to each list $S$ is its variation factor $\vf(S)$. Loosely speaking, this is the greatest number of times that $\gamma_S$ crosses any line in the plane.
To make this more precise we need the concept of a crossing segment.

\begin{definition}\label{crossing-defn}
Suppose that $\ell$ is a line in the plane and that $S = [\vecx_0,\vecx_1,\dots,\vecx_n]$. We say that the $j$th segment $s_j = \ls[\vecx_j,\vecx_{j+1}]$  is a \textit{crossing segment} of $S$ on $\ell$ if any one of the following holds:
\begin{enumerate}
  \item[(i)] $\vecx_{j}$ and $\vecx_{j+1}$ lie on (strictly) opposite sides of $\ell$.
  \item[(ii)] $j=0$ and $\vecx_{j} \in \ell$.
  \item[(iii)] $\vecx_{j} \not\in \ell$ and $\vecx_{j+1}\in  \ell$.
\end{enumerate}
\end{definition}

\begin{definition}\label{vf-defn}
Let $\vf(S,\ell)$ denote the number of crossing segments of $S$ on $\ell$. The \textit{variation factor} of $S$ is defined to be
 $\ds \vf(S) = \max_{\ell} \vf(S,\ell)$.
\end{definition}

Clearly $1 \le \vf(S) \le n$. For completeness, in the case that
$S =[\vecx_0]$ we set $\cvar(f, [\vecx_0]) = 0$ and let $\vf([\vecx_0],\ell) = 1$ whenever $\vecx_0 \in \ell$.

\begin{definition}\label{2d-var}
The \textit{two-dimensional variation} of a function $f : \sigma
\rightarrow \mathbb{C}$ is defined to be
\begin{equation*}
    \var(f, \sigma) = \sup_{S}
        \frac{ \cvar(f, S)}{\vf(S)},
\end{equation*}
where the supremum is taken over all finite ordered lists of elements of $\sigma$.
\end{definition}

The \textit{variation norm} of such a function is
  \[ \normbv{f} = \norm{f}_\infty + \var(f,\sigma) \]
and the set of functions of bounded variation on $\sigma$ is
  \[ \BV(\sigma) = \{ f: \sigma \to \mC \st \normbv{f} < \infty\}. \]
The space $\BV(\sigma)$ is a Banach algebra under pointwise operations \cite[Theorem 3.8]{AD1}.

Let $\Pol_2$ denote the space of polynomials in two real variables of the form $p(x,y) = \sum_{n,m} c_{nm} x^n y^m$, and let $\Pol_2(\sigma)$ denote the restrictions of elements on $\Pol_2$ to $\sigma$ (considered as a subset of $\R^2$). The algebra $\Pol_2(\sigma)$ is always a subalgebra of $\BV(\sigma)$ \cite[Corollary 3.14]{AD1}.

\begin{definition}
The set of \textit{absolutely continuous} functions on $\sigma$, denoted $AC(\sigma)$, is the closure of $\Pol_2(\sigma)$ in $\BV(\sigma)$.
\end{definition}

The set $AC(\sigma)$ forms a closed subalgebra of $\BV(\sigma)$ and hence is a Banach algebra.

It is an important but nonobvious fact that if $\sigma = [a,b]$ these definitions reduce  to the classical ones \cite[Proposition~3.6]{AD1}. Indeed, if $\sigma$ is any compact subset of the line, then these are the same as the $BV$ and $AC$ spaces considered by Saks \cite{S}. It \textit{is} clear  that these spaces are preserved under affine transformations of the plane. That is, if $h: \mC \to \mC$, $h(\vecz) = \alpha \vecz + \beta$ is a nontrivial affine map, then $\Phi_h(f) = f \circ h^{-1}$ defines an isometric isomorphism from $BV(\sigma)$ to $BV(h(\sigma))$. In particular $BV(\ls[\vecx,\vecy]) \simeq BV[0,1]$ for any distinct points $\vecx, \vecy \in \mC$.

More generally, if $h: \sigma_1 \to \sigma_2$ is a bijection between two subsets of $\mC$, we shall denote by $\Phi_h$ the map $f \mapsto f \circ h^{-1}$ which is always an isomorphism from the algebra of complex functions on $\sigma_1$ to the algebra of complex functions on $\sigma_2$. We shall use the same notation for the restriction of this map to any subalgebra of functions on $\sigma_1$.

One can show that $C^1(\sigma) \subseteq AC(\sigma) \subseteq C(\sigma)$, where one interprets $C^1(\sigma)$ as consisting of all functions for which there is a $C^1$ extension to an open neighbourhood of $\sigma$ (see \cite{DL2}).  We will need the following simple results which are easy consequences of the definition of $BV(\sigma)$ and $AC(\sigma)$. The (restriction to $\sigma$ of the) characteristic function of a set $A$ will be denoted $\chi_A$.

\begin{lemma}\label{Chi-BV}
For all $\vecz \in \sigma$, $\chi_{\{\vecz\}} \in BV(\sigma)$.
\end{lemma}

\begin{lemma}\label{ReIm}
If $f \in AC(\sigma)$ then $\Re f,\, \Im f \in AC(\sigma)$.
\end{lemma}

%
%

\section{Isomorphisms of $BV(\sigma)$ spaces}\label{S:Isom}

An obvious problem is to determine the nature of the possible algebra homomorphisms between two $BV(\sigma)$ spaces. As noted in the introduction, every algebra isomorphism $\Phi: AC(\sigma_1) \to AC(\sigma_2)$ is of the form $\Phi(f) = f\circ h^{-1}$ where $h: \sigma_1 \to \sigma_2$ is a homeomorphism.
It is easy to see that this result does not extend  to $BV(\sigma)$ spaces.

\begin{example}
 Let $\sigma_1 = \sigma_2 = [0,1]$ and define the bijection $h:\sigma_1 \to \sigma_2$,
  \[ h(x) = \begin{cases}
              \frac{1}{2}-x,  & \text{if $0 \le x \le \frac{1}{2}$,} \\
              x,              & \text{if $\frac{1}{2} < x \le 1$.}
  \end{cases}
  \]

A simple rearrangement of the variation of $\Phi_h(f)$ over $[0,1]$ shows that
  \[ \var(\Phi_h(f),[0,1]) \le 2 \var(f,[0,1]) \]
and so (noting that $\Phi_h^{-1} = \Phi_h$),
  \[ \frac{1}{2} \norm{f}_{BV[0,1]} \le \norm{\Phi_h(f)}_{BV[0,1]} \le 2 \norm{f}_{BV[0,1]}. \]
Thus $\Phi_h$ is a Banach algebra isomorphism from $BV(\sigma_1)$ to $BV(\sigma_2)$. But of course the map $h$ is not a homeomorphism in this case.
\end{example}

In a positive direction we have the following.

\begin{theorem}\label{BV-bijection}
Suppose that $\sigma_{1}$ and $\sigma_{2}$ are nonempty compact subsets of the plane. If $\Phi: BV(\sigma_{1})\to BV(\sigma_{2})$ is an algebra isomorphism then there exists a bijection $h:\sigma_{1}\to \sigma_{2}$ such that $\Phi(f)=f\circ h^{-1}$ for all $f\in BV(\sigma_{1})$.
\end{theorem}

\begin{proof}
Since $\Phi$ is an algebra isomorphism, it must map idempotents to idempotents. Note that by Lemma~\ref{Chi-BV}, for all $z \in \sigma_1$, the function $f_z = \chi_{\{z\}}$ lies in $BV(\sigma_1)$ and hence $g_z = \Phi(f_z)$ is an idempotent in $BV(\sigma_2)$. Since $\Phi$ is one-to-one, $g_z$ is not the zero function and hence the support of $g_z$ is a nonempty set $\tau \subseteq \sigma_2$. If $\tau$ is more than a singleton then we can choose $w \in \tau$ and write $g_z = \chi_{\{w\}} + \chi_{S\tau\setminus \{w\}}$ as a sum of two nonzero idempotents in $BV(\sigma_2)$. But then $f_z = \Phi^{-1}(\chi_{\{w\}}) + \Phi^{-1}(\chi_{S\setminus \{w\}})$ is the sum of two nonzero idempotent in $\BV(\sigma_1)$ which is impossible. It follows that $g_z$ is the characteristic function of a singleton set and this clearly induces a map $h: \sigma_1 \to \sigma_2$ so that $\Phi(f_z) = \chi_{\{h(z)\}}$. Indeed, by considering $\Phi^{-1}$ it is clear that $h$ must be a bijection between the two sets.
\end{proof}

There are several questions one might ask concerning possible strengthening of Theorem~\ref{BV-bijection}.

\begin{question}\label{quest00} Suppose that $h: \sigma_1 \to \sigma_2$ is a bijection. Does $\Phi_h$ map  $BV(\sigma_1)$ to $BV(\sigma_2)$?
\end{question}

\begin{question}\label{quest000} Suppose that $\sigma_1$ and $\sigma_2$ are homeomorphic. Is $BV(\sigma_1) \simeq BV(\sigma_2)$?
\end{question}

\begin{question}\label{quest2} Suppose that $BV(\sigma_1) \simeq BV(\sigma_2)$. Is $\sigma_1$ homeomorphic to $\sigma_2$?
\end{question}

Questions~\ref{quest00} is  easily disposed of.

\begin{example}\label{bij-not-bv} Let $\sigma_1 = \{0,1,-1,\frac{1}{2},-\frac{1}{2},\frac{1}{3},-\frac{1}{3},\dots\}$ and let $\sigma_2 = \{0,1,\frac{1}{2},\frac{1}{3},\dots\}$. Define $h: \sigma_1 \to \sigma_2$ by
 \[ h(x) = \begin{cases}
            0, & \text{if $x = 0$,} \\
            \frac{1}{2n-1}, & \text{if $x = -\frac{1}{n} < 0$,} \\
            \frac{1}{2n} , & \text{if $x = \frac{1}{n} > 0$.}
           \end{cases} \]
It is readily checked that $h$ is a  bijection (indeed a homeomorphism). If $f$ is the characteristic function of the positive elements of $\sigma_1$ then $f \in BV(\sigma_1)$ but $\Phi_h(f)$ is not in $BV(\sigma_2)$.
\end{example}

We can see from \cite[Corollary 5.12]{DAS}  that in fact no bijection between the two sets in Example~\ref{bij-not-bv} determines an isomorphism of the spaces of functions of bounded variation, so this gives a negative answer to Question~\ref{quest000}.

An examination of the proof of Theorem 3.1 in \cite{DL1} shows that if $h$ is any homeomorphism from the unit square to the closed unit disk, the map $\Phi_h$ must be unbounded with respect to the $BV$ norms, and hence this provides another counterexample to answer Question~\ref{quest000}.

Showing that the answer to Question~\ref{quest2} is also `no' is a little harder.

\begin{example}\label{non-homeo-ex}
Consider the two sets $\sigma$ and $\tau$ shown in Figure~\ref{non-homeo}. These sets are clearly not homeomorphic. Let $h$ be a bijection which maps the blue path from $\veca$ to $\vecb$ onto the closed line segment $[\vecalpha,\vecbeta]$ and the half-open line segment from $\vecc$ to $\vecd$ onto the half-open line segment from $\vecbeta$ to $\vecdelta$.

Both $\sigma$ and $\tau$ are examples of what are called linear graphs in \cite{ASD}. The algebras of functions of bounded variation on such sets admit an equivalent linear graph norm. In this case if $f: \sigma \to \mC$, then $\norm{f}_{LG(\sigma)} = \norm{f}_\infty + \var(f,\ls[\veca,\vecc]) + \var(f,\ls[\vecc,\vecb]) + \var(f,\ls[\vecc,\vecd])$ while if $g: \tau \to \mC$,
   \[ \norm{g}_{LG(\tau)}
   = \norm{g}_\infty + \var(g,\ls[\vecalpha,\vecdelta])
   = \norm{g}_\infty + \var(g,\ls[\vecalpha,\vecgamma]) + \var(g,\ls[\vecgamma,\vecbeta]) + \var(g,\ls[\vecbeta,\vecdelta]). \]
The important fact proved in \cite[Theorem 3]{ASD} is that the linear graph norm is always equivalent (as a Banach space norm) to the $BV$ norm.

\begin{figure}[H]
\begin{center}
\begin{tikzpicture}[scale=2]
 \draw[red, ultra thick]  (0,1) -- (0,0);
 \draw[blue, ultra thick] (-0.866,-0.5) -- (0,0) -- (0.866,-0.5);
 \draw (-0.866,-0.5) node[below] {$\veca$};
 \draw (0.866,-0.5) node[below] {$\vecb$};
 \draw (0,0.1) node[right] {$\vecc$};
 \draw (0,1) node[right] {$\vecd$};

\draw[black] (0,1) node[circle, draw, fill=black!50,inner sep=0pt, minimum width=4pt] {};

 \draw[black] (-0.866,-0.5) node[circle, draw, fill=black!50,inner sep=0pt, minimum width=4pt] {};

 \draw[black] (0,0) node[circle, draw, fill=black!50,inner sep=0pt, minimum width=4pt] {};

 \draw[black] (0.8666,-0.5) node[circle, draw, fill=black!50,inner sep=0pt, minimum width=4pt] {};
 \draw (0,-1) node {$\sigma$};
 \end{tikzpicture}
 \hspace{2.5cm}
 \begin{tikzpicture}
 \draw[red, ultra thick] (4,0.1) -- (6,0.1);
 \draw[blue, ultra thick] (2,0.1) -- (4,0.1);
 \draw (2,0.1) node[below] {$\vecalpha=h(\veca)$};
 \draw (3,0.1) node[above] {$\vecgamma=h(\vecc)$};
 \draw (4,0.1) node[below] {$\vecbeta=h(\vecb)$};
 \draw (6,0.1) node[below] {$\vecdelta=h(\vecd)$};

 \draw[black] (2,0.1) node[circle, draw, fill=black!50,inner sep=0pt, minimum width=4pt] {};

 \draw[black] (3,0.1) node[circle, draw, fill=black!50,inner sep=0pt, minimum width=4pt] {};

 \draw[black] (4,0.1) node[circle, draw, fill=black!50,inner sep=0pt, minimum width=4pt] {};

 \draw[black] (6,0.1) node[circle, draw, fill=black!50,inner sep=0pt, minimum width=4pt] {};
  \draw (3.9,-1) node {$\tau$};

\end{tikzpicture}
\end{center}
\caption{Two non-homeomorphic compact sets with $BV(\sigma) \simeq BV(\tau)$.}\label{non-homeo}
\end{figure}
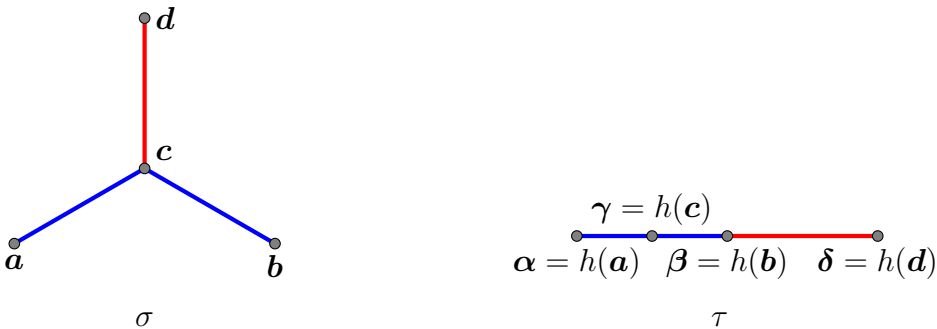

Suppose that $f \in BV(\sigma)$ and that $g = \Phi_h(f) = f \circ h^{-1} \in BV(\tau)$. It is clear that
$\norm{f}_\infty = \norm{g}_\infty$, that $\var(f,\ls[\veca,\vecc]) = \var(g,\ls[\vecalpha,\vecgamma])$ and that $\var(f,\ls[\vecc,\vecb]) = \var(g,\ls[\vecgamma,\vecbeta])$. One may verify that
  \[ \var(g,\ls[\vecbeta,\vecdelta]) \le \var(f,\ls[\vecc,\vecb]) + \var(f,\ls[\vecc,\vecd]) \]
and consequently $\norm{g}_{LG(\tau)} \le 2 \norm{f}_{LG(\sigma)}$. An analogous calculation also shows that $\norm{f}_{LG(\sigma)} \le 2\norm{g}_{LG(\tau)} $.

Using the fact that the $LG$ norms are equivalent to the $BV$ norms. this implies that the algebra isomorphism $\Phi_h$ is continuous, with continuous inverse, from $BV(\sigma)$ to $BV(\tau)$, and hence these spaces are isomorphic as Banach algebras.

Note that since  $\sigma$ is not homeomorphic
to $\tau$ we know that $AC(\sigma) \not\simeq AC(\tau )$.
\end{example}

What really determines whether a bijection defines an isomorphism of the $BV(\sigma)$ spaces is what it does to the variation factors of lists of points.

\begin{definition} Suppose that $h: \sigma_1 \to \sigma_2$ is a bijection.
\begin{enumerate}
\item If $S = [\vecx_0,\vecx_1,\dots,\vecx_n]$ is a finite list of elements of $\sigma_1$, we denote
the corresponding list of elements in $\sigma_1$
by $h(S) = [h(\vecx_0),h(\vecx_1),\dots,h(\vecx_n)]$.
\item The \textbf{variation factor of $h$} is
  \[ \vf(h) = \sup_S \frac{\vf(S)}{\vf(h(S))}. \]
\end{enumerate}
\end{definition}

Note that $\vf(h)$ is always at least $1$, and that $\vf(h)$ may be infinite.

\begin{lemma}\label{bv-boundedness} Suppose that $h: \sigma_1 \to \sigma_2$ is a bijection and that $\vf(h) = K< \infty$.
If $f \in BV(\sigma_1)$ then $\Phi_h(f) \in BV(\sigma_2)$ and
   \[ \norm{\Phi_h(f)}_{BV(\sigma_2)} \le K \norm{f}_{BV(\sigma_1)}. \]
Equivalently, if $\Phi_h$ is not bounded on $BV(\sigma_1)$ then $\vf(h) = \infty$.
\end{lemma}

\begin{proof} Let ${\hat S}$ be a finite list of points in $\sigma_2$. As $h$ is a bijection, there exists a finite list $S$ in $\sigma_1$ such that ${\hat S} = h(S)$. Then
  \[ \frac{\cvar(\Phi_h(f),{\hat S})}{\vf({\hat S})}
      = \frac{\cvar(f,S)}{\vf(h(S))}
      \le \frac{K \cvar(f,S)}{\vf(S)}
      \le K \var(f,\sigma_1).\]
Since $\norm{\Phi_h(f)}_\infty = \norm{f}_\infty$, the result follows.
\end{proof}

\begin{lemma}\label{vf(h)-lem} If $\vf(h) = \infty$ then $\Phi_h$ is not bounded on $BV(\sigma_1)$ and hence $\Phi_h$ does not map $BV(\sigma_1)$ onto $BV(\sigma_2)$.
\end{lemma}

\begin{proof}
Suppose that $\vf(h) = \infty$ and that $K > 1$. Then there exists a finite list $S = [\vecx_0,\dots,\vecx_n]$ in $\sigma_1$ such that $\frac{\vf(S)}{\vf(h(S))} > K$.
Suppose that $\vf(S) = m$, and note that $m > 1$. Choose any line $\ell$ such that $\vf(S,\ell) = \vf(S)$.
This line determines two closed half-planes whose intersection is $\ell$. By \cite[Proposition~3.20]{AD1}, the characteristic functions of these half planes, $\chi_1$ and $\chi_2$, are of bounded variation on $\sigma_1$ with $\norm{\chi_i}_{BV(\sigma_1)} \le 2$.

By definition, $m$ of the segments in $S$ are crossing segments of $S$ on $\ell$. Of these, at least $m-1$ must satisfy either rule (i) or rule (iii) of Definition~\ref{crossing-defn}. If $\ls[\vecx_{j},\vecx_{j+1}]$ satisfies rule (i), then $|\chi_i(\vecx_{j+1}) - \chi_i(\vecx_{j})| = 1$ for each $i$. If $\ls[\vecx_{j},\vecx_{j+1}]$ satisfies rule (iii), then $|\chi_i(\vecx_{j+1}) - \chi_i(\vecx_{j})| = 1$ for one value of $i$. Combining these facts shows that for at least one of the value $i=1$ or $i=2$,
  \begin{equation}\label{pick-i} \sum_{j=1}^n  |\chi_i(\vecx_j) - \chi_i(\vecx_{j-1})| \ge \frac{m-1}{2}. \end{equation}

Fix $i$ so that (\ref{pick-i}) holds. Then, as $m > 1$,
  \[ \var(\Phi_h(f))
     \ge \frac{\cvar(\Phi_h(f),h(S))}{\vf(h(S))}
        =\frac{\cvar(f,S)}{\vf(h(S))}
        \ge \frac{m-1}{2} \cdot \frac{K}{m}
        \ge \frac{K}{4} .
   \]
Thus
   \[ \norm{\Phi_h} \ge \frac{\norm{\Phi_h(f)}_{BV(\sigma_2)}}{\norm{f}_{BV(\sigma_1)}} \ge \frac{K+1}{8}. \]
Since $K$ was arbitrary, $\Phi_h$ is unbounded. The final conclusion of the lemma follows from the Banach Isomorphism Theorem.
\end{proof}

\begin{theorem}\label{h-vf-cond} Suppose that $h: \sigma_1 \to \sigma_2$ is a bijection. Then $\Phi_h$ is a (Banach algebra) isomorphism from $BV(\sigma_1)$ to $BV(\sigma_2)$ if and only if $\vf(h)$ and $\vf(h^{-1})$ are both finite.
\end{theorem}

\begin{proof} $\Phi_h$ is always an algebra isomorphism from the algebra of all complex-valued functions on $\sigma_1$ to the algebra of all functions on $\sigma_2$. By the previous lemmas, the boundedness of $\Phi_h$ and $\Phi_h^{-1} = \Phi_{h^{-1}}$ is equivalent to the conditions that $\vf(h)$ and $\vf(h^{-1})$ are finite.
\end{proof}

%
%

\section{$AC(\sigma)$ and $BV(\sigma)$ spaces}\label{S:ACBV}

The isomorphisms between pairs of $AC(\sigma)$ spaces constructed in \cite{ASD,DAS,DL1} are all restrictions of isomorphisms of the corresponding $BV(\sigma)$ spaces. As we have seen in Example~\ref{non-homeo-ex}, not every isomorphism of $BV(\sigma)$ spaces preserves the subalgebras of absolutely continuous functions.

Two natural questions arise.

\begin{question}\label{quest3} Suppose that $BV(\sigma_1) \simeq BV(\sigma_2)$ via the isomorphism $\Phi_h(f) = f \circ h^{-1}$ where $h: \sigma_1 \to \sigma_2$ is a homeomorphism. Does $\Phi_h$ map $AC(\sigma_1)$ to $AC(\sigma_2)$?
\end{question}

\begin{question}\label{quest4} Suppose that $\Phi_h$ is an isomorphism from $AC(\sigma_1)$ to $AC(\sigma_2)$. Does $\Phi_h$ extend to an isomorphism from $BV(\sigma_1)$ to $BV(\sigma_2)$?
\end{question}

Question~\ref{quest3} is easily dealt with.

\begin{lemma}\label{h-ac}
Suppose that $\sigma_1, \sigma_2 \subseteq \mathbb{C}$ and that $h: \sigma_1 \to \sigma_2$ is a homeomorphism. If $\Phi_h$ is an isomorphism from $AC(\sigma_1)$ to $AC(\sigma_2)$ then $h \in AC(\sigma_1)$ and $h^{-1} \in AC(\sigma_2)$.
\end{lemma}

\begin{proof}
The map identity map $f(\vecx) = \vecx$ is clearly in $AC(\sigma_1)$ and so $\Phi_h(f) = h^{-1} \in AC(\sigma_2)$. The same argument applied to $\Phi^{-1}$ shows that $h \in AC(\sigma_1)$.
\end{proof}

\begin{example}
Let $\sigma_1 = \sigma_2 = [0,1]$. Let $h: \sigma_1 \to \sigma_2$ be an increasing bijection which is not absolutely continuous. (For example, $h(x) = \frac{1}{2}(x+C(x))$ where $C$ is the Cantor function.) Then $\Phi_h$ is an isomorphism from $BV(\sigma_1)$ to $BV(\sigma_2)$, but, by the lemma, it does not map $AC(\sigma_1)$ to $AC(\sigma_2)$. It follows that the answer to Question~\ref{quest3}  is `no'.
\end{example}

The following fact will be needed in the proof of the main result, Theorem~\ref{ac-bv-extend}.

\begin{lemma}\label{intoAC}
Suppose that $\sigma_1, \sigma_2 \subseteq \mathbb{C}$ and that $h: \sigma_1 \to \sigma_2$ is a homeomorphism such that $h^{-1} \in AC(\sigma_2)$. Then
for any $p \in \Pt_2(\sigma_1)$, $\Phi_h(p) \in AC(\sigma_2)$.
\end{lemma}

\begin{proof}
Let $p_x(x,y) = x$ and $p_y(x,y) = y$. Then $\Phi_h(p_x) = p_x \circ h^{-1} = \Re h^{-1}$ which lies in $AC(\sigma_2)$ by Lemma~\ref{ReIm}. Similarly $\Phi_h(p_y) \in AC(\sigma_2)$ and so the algebraic properties of $\Phi_h$ ensure that $\Phi_h(p) \in AC(\sigma_2)$ for all $p \in \Pt_2(\sigma_1)$.
\end{proof}

\begin{theorem}\label{ac-bv-extend} Suppose that $AC(\sigma_1) \simeq AC(\sigma_2)$. Then $BV(\sigma_1) \simeq BV(\sigma_2)$.
\end{theorem}

\begin{proof} Let $\Phi: AC(\sigma_1) \to AC(\sigma_2)$ be an algebra isomorphism from $AC(\sigma_1)$ to $AC(\sigma_2)$ and let $K = \norm{\Phi}$. Then \cite[Theorem~2.8]{DL1} implies that there exists a homeomorphism $h: \sigma_1 \to \sigma_2$ such that $\Phi(f) = \Phi_h(f) = f \circ h^{-1}$ for all $f \in AC(\sigma_1)$.

Suppose that $\Phi_h$ is  not bounded as a map on $BV(\sigma_1)$. By
Lemma~\ref{bv-boundedness} this means that $\vf(h) = \infty$. We can therefore choose a list of points $S \subseteq \sigma_1$ so that $\vf(S) > 8K \vf(h(S))$. As in the proof of Lemma~\ref{vf(h)-lem}, we can choose a line $\ell$ such that $\vf(S,\ell) = \vf(S) = m$ say.  This line forms the boundary of a closed
half-plane $H$ whose characteristic function $\chi_H$ satisfies
$\cvar(\chi_H,S) \ge \frac{m-1}{2}$.

As in Section~9 of \cite{ASD4} we can choose a piecewise planar function $f_{H,\delta}\in AC(\sigma_1)$ of norm at most $2$ which agrees with $\chi_H$ except on a small strip of width $\delta$ along the boundary of $H$ (see Figure~\ref{f_H,delta-pic}).
Since $S$ is a finite set, if $\delta$ is chosen small enough, then $f_{H,\delta}$ and $\chi_H$ agree on the points in $S$ and hence
  \[ \cvar(\Phi_h(f_{H,\delta}),h(S)) = \cvar(f_{H,\delta},S) = \cvar(\chi_H,S) \ge \frac{m-1}{2}. \]
Thus
   \begin{align*}
    \var(\Phi_h(f_{H,\delta}),\sigma_2)
      &\ge \frac{\cvar(\Phi_h(f_{H,\delta}),h(S))}{\vf(h(S))}  \\
     & > \frac{m-1}{2} \cdot \frac{8K}{\vf(S)} \\
     &= 4 \frac{(m-1)K}{m} \ge 2K.
  \end{align*}
We then have that
  \[ \norm{\Phi_h(f_{H,\delta})}_{BV(\sigma_2)} > 2K \ge \norm{\Phi_h} \, \norm{f_{H,\delta}}_{BV(\sigma_1)} \]
which is impossible. Therefore $\vf(h)$ must be finite, and $\Phi_h$ must be bounded on $BV(\sigma_2)$.

An analogous argument using the relationship between the boundedness of $\Phi_h^{-1}$ and the finiteness of $\vf(h^{-1})$ completes the proof.
\end{proof}

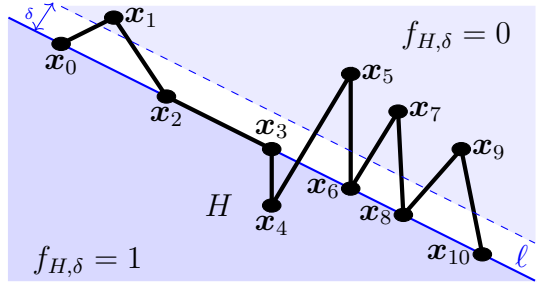
\begin{figure}[ht!]
\begin{center}

\begin{tikzpicture}[yscale=2.5,xscale=3.5]
  \draw[fill,blue!15] (-1,0.7) -- (-1,-0.7) -- (1,-0.7) -- (-1,0.7);
  \draw[fill,blue!8] (-0.8,0.76) -- (1,-0.5) -- (1,0.76) -- (-0.8,0.76);
  \draw[thick,blue] (-1,0.7) -- (1,-0.7);
  \draw[ultra thick, black] (-0.8,0.56) -- (-0.6,0.7) -- (-0.4,0.28) -- (0,0) -- (0,-0.3) -- (0.3,0.4) -- (0.3,-0.21) -- (0.48,0.2) -- (0.5,-0.35) -- (0.72,0) -- (0.8,-0.56);
  \draw[blue,dashed] (-0.8,0.76) -- (1,-0.5);
  \draw[<->,blue] (-0.9,0.63) -- (-0.83,0.78);
  \draw[blue] (-0.98,0.72) node[right] {\tiny{$\delta$}};
  \draw (-0.7,-0.6) node {$f_{H,\delta} = 1$};
  \draw (0.7,0.6) node {$f_{H,\delta} = 0$};

   \draw[fill,black] (-0.8,0.56) circle (1pt) node[below] {$\vecx_0$};
   \draw[fill,black] (-0.6,0.7) circle (1pt) node[right] {$\vecx_1$};
   \draw[fill,black] (-0.4,0.28) circle (1pt) node[below] {$\vecx_2$};
   \draw[fill,black] (0,0) circle (1pt) node[above] {$\vecx_3$};
   \draw[fill,black] (0,-0.3) circle (1pt) node[below] {$\vecx_4$};
   \draw[fill,black] (0.3,0.4) circle (1pt) node[right] {$\vecx_5$};
   \draw[fill,black] (0.3,-0.21) circle (1pt) node[left] {$\vecx_6$};
   \draw[fill,black] (0.48,0.2) circle (1pt) node[right] {$\vecx_7$};
   \draw[fill,black] (0.5,-0.35) circle (1pt) node[left] {$\vecx_8$};
   \draw[fill,black] (0.72,0.0) circle (1pt) node[right] {$\vecx_9$};
   \draw[fill,black] (0.8,-0.56) circle (1pt) node[left] {$\vecx_{10}$};

   \draw[blue] (0.95,-0.57) node {$\ell$};
   \draw (-0.2,-0.3) node {$H$};
\end{tikzpicture}
\caption{Choosing the  function $f_{H,\delta}\in AC(\sigma_1)$ in the proof of Theorem~\ref{ac-bv-extend}. For $\delta$ small enough, $f_{H,\delta}(\vecx_j) = \chi_H(\vecx_j)$ for all $j$.}\label{f_H,delta-pic}
\end{center}
\end{figure}

Combining the above results gives the following characterization.

\begin{theorem}
Suppose that $\sigma_1, \sigma_2 \subseteq \mathbb{C}$ and that $h: \sigma_1 \to \sigma_2$ is a homeomorphism. Then the following are equivalent.
\begin{enumerate}
  \item $\Phi_h$ is an isomorphism from $AC(\sigma_1)$ to $AC(\sigma_2)$.
  \item $\Phi_h$ is an isomorphism from $BV(\sigma_1)$ to $BV(\sigma_2)$, $h \in AC(\sigma_1)$ and $h^{-1} \in AC(\sigma_2)$
\end{enumerate}
\end{theorem}

\begin{proof} ($\implies$) Suppose that $\Phi_h$ is an isomorphism from $AC(\sigma_1)$ to $AC(\sigma_2)$. By the previous theorem, $\Phi_h$ extends to an isomorphism from $BV(\sigma_1)$ to $BV(\sigma_2)$. The facts about $h$ and $h^{-1}$ follow from Lemma~\ref{h-ac}.

($\Leftarrow$) Suppose that (2) holds. By Lemma~\ref{intoAC}, $\Phi_h(p) \in AC(\sigma_2)$ for all $p \in \Pt_2(\sigma_1)$. Since $\Phi_h$ is $BV$ norm bounded this implies that $\Phi_h(f) \in AC(\sigma_2)$ for all $f \in AC(\sigma_1)$. Similarly $\Phi_h^{-1}$ maps $AC(\sigma_2)$ into $AC(\sigma_1)$ and hence $\Phi_h$ is an isomorphism of the spaces of absolutely continuous functions.
\end{proof}

\section{An application}

Berkson and Gillespie \cite{BG} introduced the class of trigonometrically well-bounded operators as a type of Banach space analogue of unitary operators. Trigonometrically well-bounded operators have been used in developing various aspects of operator-valued harmonic analysis (see for example \cite{BG2}). It was shown in \cite{AD4} that for reflexive Banach spaces, these operators are precisely the $AC(\mT)$ operators, that is, operators which possess an $AC(\mT)$ functional calculus. (On nonreflexive spaces, the definition of a trigonometrically well-bounded operator  requires that this functional calculus is weakly compact.) Berkson and Gillespie showed that if $T$ is trigonometrically well-bounded then it admits an integral representation with respect to a suitable family of projections, and this can be used to extend the $AC(\mT)$ functional calculus to a $BV(\mT)$ functional calculus. An open problem is whether every $AC(\sigma)$ operator on a reflexive Banach space admits a $BV(\sigma)$ functional calculus.

It is not true that $AC(\sigma)$ is isomorphic to $AC(\mT)$ for every set $\sigma \subseteq \mC$ which is homeomorphic to $\mT$. However one can obtain a Gelfand--Kolmogorov type theorem if one restricts the class of sets somewhat. In \cite{ASD4} the authors introduce the family $PIC$ of `polygonally inscribed curves'. These are connected sets which can be written as a finite union of smooth convex curves, subject to some mild conditions about how these curves meet. $PIC$ contains all the linear graph sets considered in $\cite{ASD}$. We omit the full definition here, but some examples of $PIC$ sets are given in Figure~\ref{PIC-sets}.

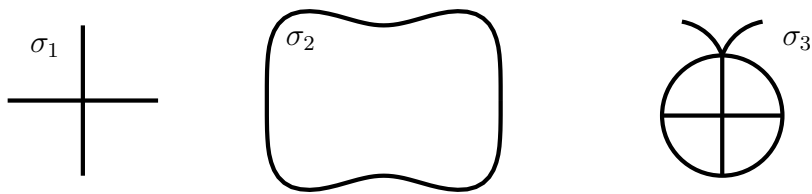
\begin{figure}[ht!]
\begin{center}
\begin{tikzpicture}

\draw[ultra thick, black] (-4,-1) -- (-4,1);
\draw[ultra thick, black] (-5,0) -- (-3,0);
\draw (-4.5,0.7) node {$\sigma_1$};

\draw[ultra thick, black] (1.554, 0.) -- (1.554, 0.09776) -- (1.554, .1962) -- (1.553, .2963) -- (1.551, .3982) -- (1.546, .5024) -- (1.538, .6088) -- (1.523, .7166) -- (1.497, .8229) -- (1.458, .9253) -- (1.401, 1.018) -- (1.323, 1.094) -- (1.225, 1.150) -- (1.109, 1.181) -- (.9835, 1.189) -- (.8564, 1.179) -- (.7335, 1.156) -- (.6206, 1.129) -- (.5178, 1.100) -- (.4252, 1.074) -- (.3414, 1.051) -- (.2651, 1.033) -- (.1945, 1.018) -- (.1273, 1.008) -- (0.06304, 1.002) -- (0, 1.000) -- (-0.06304, 1.002) -- (-.1273, 1.008) -- (-.1941, 1.018) -- (-.2651, 1.033) -- (-.3414, 1.051) -- (-.4252, 1.074) -- (-.5178, 1.100) -- (-.6206, 1.129) -- (-.7335, 1.156) -- (-.8564, 1.179) -- (-.9835, 1.189) -- (-1.109, 1.181) -- (-1.225, 1.150) -- (-1.323, 1.095) -- (-1.401, 1.018) -- (-1.458, .9255) -- (-1.497, .8227) -- (-1.523, .7163) -- (-1.538, .6083) -- (-1.546, .5031) -- (-1.551, .3985) -- (-1.553, .2964) -- (-1.554, .1962) -- (-1.554, 0.09739) -- (-1.554, -0.0006329) -- (-1.554, -0.9711e-1) -- (-1.554, -.1959) -- (-1.553, -.2961) -- (-1.551, -.3983) -- (-1.546, -.5029) -- (-1.538, -.6095) -- (-1.523, -.7159) -- (-1.497, -.8224) -- (-1.458, -.9253) -- (-1.401, -1.018) -- (-1.323, -1.095) -- (-1.225, -1.151) -- (-1.109, -1.180) -- (-.9835, -1.189) -- (-.8564, -1.179) -- (-.7335, -1.156) -- (-.6206, -1.129) -- (-.5178, -1.100) -- (-.4252, -1.074) -- (-.3414, -1.051) -- (-.2651, -1.033) -- (-.1941, -1.018) -- (-.1273, -1.008) -- (-0.6304e-1, -1.002) -- (0, -1.000) -- (0.06304, -1.002) -- (.1273, -1.008) -- (.1941, -1.018) -- (.2651, -1.032) -- (.3414, -1.051) -- (.4252, -1.074) -- (.5178, -1.100) -- (.6206, -1.129) -- (.7335, -1.156) -- (.8564, -1.178) -- (.9835, -1.188) -- (1.109, -1.181) -- (1.225, -1.150) -- (1.323, -1.095) -- (1.401, -1.018) -- (1.458, -.9250) -- (1.497, -.8221) -- (1.523, -.7171) -- (1.538, -.6092) -- (1.546, -.5024) -- (1.551, -.3980) -- (1.553, -.2958) -- (1.554, -.1956) -- (1.554, -0.09831) -- (1.554, -0.0002880);
\draw (-1.1,0.8) node {$\sigma_2$};

\draw[ultra thick,black] (4.5,-0.2) circle (0.8cm);
\draw[ultra thick,black] (4.5,-1) -- (4.5,0.6);
\draw[ultra thick,black] (3.7,-0.2) -- (5.3,-0.2);
\draw[ultra thick, black] (4.5,0.6) arc(20:80:0.7);
\draw[ultra thick, black] (4.5,0.6) arc(160:100:0.7);
\draw (5.5,0.8) node {$\sigma_3$};

\end{tikzpicture}
\end{center}
\caption{Three polygonally inscribed curves}\label{PIC-sets}

\end{figure}

\begin{theorem} Suppose that $T$ is an $AC(\sigma)$ operator on a reflexive Banach space $X$ with $\sigma \in PIC$. If $\sigma$ is homeomorphic to $\mT$ then $T$ admits a $BV(\sigma)$ functional calculus.
\end{theorem}

\begin{proof}
Let $\Psi_T: AC(\sigma) \to B(X)$ denote the functional calculus homomorphism for $T$. Since $\sigma$ and $\mT$ are homeomorphic sets in $PIC$, it follows from \cite[Theorem~7]{ASD4} that there exists an isomorphism $\Phi: AC(\sigma) \to AC(\mT)$. Define $\Gamma: AC(\mT) \to B(X)$ by $\Gamma = \Psi_T \circ \Phi^{-1}$. Let $e(z) = z$ be the identity function considered as an element of $AC(\mT)$ and let $U = \Gamma(e)$. Then $\Gamma$ determines an $AC(\mT)$ functional calculus for $U$, and so $U$ is trigonometrically well-bounded. It follows that $\Gamma$ extends to a bounded algebra homomorphism ${\hat \Gamma}: BV(\mT) \to B(X)$. By Theorem~\ref{ac-bv-extend}, $\Phi$ lifts to an isomorphism ${\hat \Phi}: BV(\sigma) \to BV(\mT)$. Let ${\hat \Psi}_T: BV(\sigma) \to B(X)$ be ${\hat \Psi}_T = {\hat \Gamma} \circ {\hat \Phi}$. Then ${\hat \Psi}_T$ defines a $BV(\sigma)$ functional calculus for $T$.
\end{proof}

\section{Appendix: The definition of variation}

The definition of $\var(f,\sigma)$ has appeared in various guises since it was introduced in \cite{AD1}, and it is reasonable to ask whether all the versions used are equal. The aim of this appendix is to confirm that the various forms of the definition are indeed consistent.

In its original form the variation was defined to be
  \[ \var(f,\sigma) = \sup_{\gamma \in \Gamma} \frac{\cvar(f,\gamma)}{\vf(\gamma)} \]
where $\Gamma$ was the space of continuous curves in the plane parametrized by $[0,1]$. Even in \cite{AD1} it was noted that working with the space $\Gamma$ was unwieldy and that it was sufficient to consider the space $\Gamma_L$ of piecewise linear curves. Indeed the proofs of the properties of variation in \cite{AD1} all utilized $\Gamma_L$ in place of $\Gamma$. Each $\gamma \in \Gamma_L$ can be specified by giving an ordered list $S = [\vecx_0,\vecx_1,\dots,\vecx_n]$ of points so that $\gamma$ is made up of the line segments $\ls[\vecx_{j},\vecx_{j+1}]$, $j = 0,\dots,n-1$.

The main issue here is the definition of $\vf(\gamma)$, which heuristically aims to count the maximum number of times that $\gamma$ crosses any line. The challenge was to make sense of what counts as a crossing.
The original definition in \cite{AD1} involved entry points of the curve $\gamma$ on a line $\ell$ in the plane. If a continuous curve $\gamma$ is parametrized as $\gamma(t)$, $0 \le t \le 1$, then $t$ is an \textbf{entry point} of $\gamma$ on $\ell$ if either:
\begin{enumerate}
  \item[(1)] $t= 0$ and $\gamma(0) \in \ell$, that is the curve starts on $\ell$, or
  \item[(2)] $0 < t \le 1$, $\gamma(t) \in \ell$ and for all $u \in (0,t)$ there exists $s \in (u,t)$ such that $\gamma(s) \not\in \ell$.
\end{enumerate}
Then $\vf(\gamma,\ell)$ was defined to be the number of entry points of $\gamma$ on $\ell$ and $\vf(\gamma)$ was defined to be the maximum value of $\vf(\gamma,\ell)$ over all lines $\ell$ in the plane.

Suppose that $\ell$ is a line and that $\gamma = \gamma_S$ is a piecewise linear curve determined by $S = [\vecx_0,\dots,\vecx_n]$, so that $\vecx_j = \gamma(t_j)$ with $0 = t_0 < t_1 < \dots < t_n = 1$. The core observation is that for each $j \in \{0,1,\dots,n-1\}$ there can be at most one entry point of $\gamma$ on $\ell$ in $[t_{j},t_{j+1}]$. To see this, note that otherwise there would be two points in $\seg(j,j+1)$ on $\ell$ and hence $\seg(j,j+1) \subseteq \ell$. But by (2) this would mean that no $t \in (t_j,t_{j+1}]$ is an entry point.

This led to the concept of a crossing segment, introduced in \cite{DL1,DL2}. There, a segment $\ls[\vecx_j,\vecx_{j+1}]$ was called a crossing segment of $S$ on $\ell$ if either $\gamma$ has an entry point in $[t_{j},t_{j+1})$, or, to deal with the final endpoint, if $j = n-1$ and $t_n = 1$ is an entry point. Note that if $j = n-1$ and $t_n$ is an entry point, then $\gamma$ does not have an entry point in $[t_{n-1},t_n)$.
 Thus the number of crossing segments is equal to the number of entry points for $\gamma$.

The aim of introducing crossing segments was to avoid dealing with parameterizations entirely and to express (1) and (2) in terms of the points in $S$.
There are a number of cases:
\begin{itemize}
  \item[(a)] If $t \in (t_{j},t_{j+1})$ is an entry point, then $\vecx_{j}$ and $\vecx_{j+1}$ must lie on opposite sides of $\ell$.
  \item[(b)] Otherwise $t_{j}$ is an entry point for some $j$.
  \begin{enumerate}
    \item[(b1)] If $j = 0$ this means that $\vecx_0 \in \ell$.
    \item[(b2)] If $j > 0$ then $x_j \in \ell$ but $x_{j-1} \not\in \ell$ (or else we would have $\seg(j-1,j) \subseteq \ell$).
   \end{enumerate}
\end{itemize}

Encoding this in terms of the endpoints of the segments of $S$ gave the following definition from \cite{DL1}. (Again it is worth noting that if $t_0$ is an entry point then there is no entry point in $(t_0,t_1]$.)

\begin{definition}\label{olddef}
The $j$th segment $s_j = \ls[\vecx_j,\vecx_{j+1}]$ is a \textbf{crossing segment} of $S$ on the line $\ell$ if any of the following hold.
  \begin{itemize}
  \item[(I)] $\vecx_{j}$ and $\vecx_{j+1}$ lie on (strictly) opposite sides of $\ell$.
  \item[(II)] $j=0$ and $\vecx_{j} \in \ell$.
  \item[(III)] $j > 0$, $\vecx_j \in \ell$ and $\vecx_{j-1} \not\in \ell$.
  \item[(IV)] $j = n-1$, $\vecx_j \not\in \ell$ and $\vecx_{j+1} \in \ell$.
  \end{itemize}
\end{definition}

Condition (I) corresponds to condition (a) above. Condition (II) is (b1) when $j = 0$. Condition (III) is (b2) for $0 < j < n$. Condition (IV) covers the case when the final point (that is $t_{n}$) is an entry point.

More recently it was realized that it was in fact simpler to count the cases where either $t_0$ is an entry point, or else where there is an entry point in $(t_j,t_{j+1}]$. This leads to the more elegant definition used in \cite{ASD4}.

\begin{definition}\label{newdef}
The $j$th segment $s_j = \ls[\vecx_j,\vecx_{j+1}]$ is a \textbf{crossing segment} of $S$ on the line $\ell$ if any of the following hold.
  \begin{itemize}
  \item[(i)] $\vecx_{j}$ and $\vecx_{j+1}$ lie on (strictly) opposite sides of $\ell$.
  \item[(ii)] $j=0$ and $\vecx_{j} \in \ell$.
  \item[(iii)] $\vecx_{j} \not\in \ell$ and $\vecx_{j+1}\in  \ell$.
  \end{itemize}
\end{definition}

Note that Definition~\ref{newdef} may label different segments as crossing segments as compared to Definition~\ref{olddef}, but it will always produce the same number of crossing segments. This is best illustrated by an example.

\begin{example} Consider the line $\ell$ and the ordered list $S = [\vecx_0,\vecx_1,\dots,\vecx_9]$ shown in Figure~\ref{crossings}. The locations of the five entry points of the curve $\gamma$ determined by $S$ are marked in blue. Definition~\ref{olddef} will count five crossing segments: $\seg(0,1)$ (rule (II)), $\seg(2,3)$ (rule (I)), $\seg(4,5)$ (rule (III)), $\seg(7,8)$ (rule (I)) and $\seg(8,9)$ (rule (IV)). Definition~\ref{newdef} also counts five crossing segments: $\seg(0,1)$ (rule (ii)), $\seg(2,3)$ (rule (i)), $\seg(3,4)$ (rule (iii)), $\seg(7,8)$ (rule (i)) and $\seg(8,9)$ (rule (iii)).

\begin{figure}[!ht]
\begin{center}
\begin{tikzpicture}[scale=1]
 \draw[thick,green] (-1,0) --(11,0);
 \draw[ultra thick, red] (1,0) -- (2,0) -- (3,1) -- (4,-1) -- (5,0) -- (6,0) -- (8,0) -- (7,1) -- (7,-1) -- (9,0);
\draw[blue] (1,0) node[circle, draw, fill=blue!50,inner sep=0pt, minimum width=6pt] {};
\draw (1,0) node[above] {$\vecx_0$};
\draw[red] (2,0) node[circle, draw, fill=black!50,inner sep=0pt, minimum width=4pt] {};
\draw (2,0) node[below] {$\vecx_1$};
\draw[red] (3,1) node[circle, draw, fill=black!50,inner sep=0pt, minimum width=4pt] {};
\draw (3,1) node[above] {$\vecx_2$};
\draw[blue] (3.5,0) node[circle, draw, fill=blue!50,inner sep=0pt, minimum width=6pt] {};
\draw (3.65,0) node[above] {$\vecv$};
\draw[red] (4,-1) node[circle, draw, fill=black!50,inner sep=0pt, minimum width=4pt] {};
\draw (4,-1) node[below] {$\vecx_3$};
\draw[blue] (5,0) node[circle, draw, fill=blue!50,inner sep=0pt, minimum width=6pt] {};
\draw (5,0) node[above] {$\vecx_4$};
\draw[red] (6,0) node[circle, draw, fill=black!50,inner sep=0pt, minimum width=4pt] {};
\draw (6,0) node[above] {$\vecx_5$};
\draw[red] (8,0) node[circle, draw, fill=black!50,inner sep=0pt, minimum width=4pt] {};
\draw (8.1,0) node[above] {$\vecx_6$};
\draw[red] (7,1) node[circle, draw, fill=black!50,inner sep=0pt, minimum width=4pt] {};
\draw (7,1) node[above] {$\vecx_7$};
\draw[blue] (7,0) node[circle, draw, fill=blue!50,inner sep=0pt, minimum width=6pt] {};
\draw (6.8,0) node[below] {$w$};
\draw[red] (7,-1) node[circle, draw, fill=black!50,inner sep=0pt, minimum width=4pt] {};
\draw (7,-1) node[below] {$\vecx_8$};
\draw[blue] (9,0) node[circle, draw, fill=blue!50,inner sep=0pt, minimum width=6pt] {};
\draw (9,0) node[above] {$\vecx_9$};

 \draw (10.5,0) node[below] {$\ell$};

\end{tikzpicture}
\caption{Crossing segments of $S = [\vecx_0,\vecx_1,\dots,\vecx_n]$ on $\ell$.}\label{crossings}
\end{center}
\end{figure}
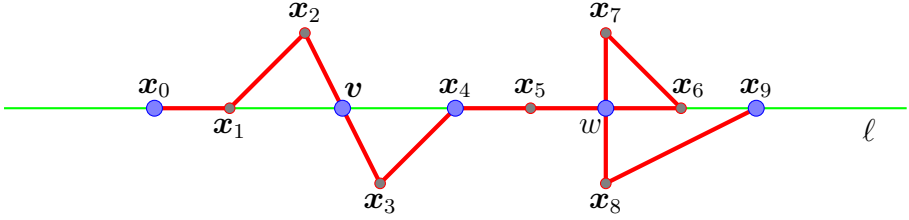

\end{example}

This means that the value of $\vf(S,\ell)$, the number of crossing segments of $S$ on $\ell$ and of the variation factor of $S$, $\vf(S) = \max_\ell \vf(S,\ell)$, are consistent between these two definitions (and the original definition in \cite{AD1}). Consequently, the definition of $\var(f,\sigma)$ is unchanged if one uses either definition for crossing segments.

An alternative way to calculate $\vf(S,\ell)$ is to split the curve $\gamma_S$ into connected sections which are either on or off the line $\ell$. (Note that connected here refers to the parameterizarion of $\gamma_S$. This is different to looking at the connected components of $\gamma_S \cap \ell$.) In Figure~\ref{crossings} the sections which are on the line are $\seg(0,1)$, $\{\vecv\}$, $\seg(4,6)$, $\{w\}$ and $\{\vecx_9\}$.

\begin{proposition}\label{conn-sect} $\vf(S,\ell)$ is the number of connected sections of $\gamma_S$ which are on $\ell$.
\end{proposition}

\begin{proof} Fix a parameterisation of $\gamma_S$ by $[0,1]$. Then there exist points
  \[ 0 \le b_1 \le e_1 < b_2 \le e_2 < \dots < b_m \le e_m \le 1 \]
so that the $i$th connected section of $\gamma_S$ that lies on $\ell$ goes from $\gamma_S(b_i)$ to $\gamma_S(e_i)$.
For each $i$, $b_i$ is an entry point of $\gamma_S$ on $\ell$. Conversely every entry point is one of the $b_i$s. Thus the number of connected sections is equal to the number of entry points, and from the earlier remarks, this is equal to $\vf(S,\ell)$.
\end{proof}

(We note that a different definition again was used in \cite{ASD}, where condition (IV) was erroneously omitted from Definition~6.1. This causes no issue in that paper as the only crossing segments that were needed there were ones which satisfied condition (I).)

It is worth noting that both Definition~\ref{olddef} and Definition~\ref{newdef} are dependent on the direction in which the curve $\gamma_S$ is traversed. Let $S_r =  [\vecx_n,\vecx_{n-1},\dots,\vecx_0]$ denote the points of $S$ listed in the reverse order. For the example in Figure~\ref{crossings}, under Definition~\ref{newdef}, $\seg(7,6)$ is a crossing segment for $S_r$ on $\ell$ while $\seg(6,7)$ is not one for $S$. By Proposition~\ref{conn-sect} however, it is clear that one always has $\vf(S,\ell) = \vf(S_r,\ell)$.

A piecewise linear curve $\gamma$ can be generated by different ordered lists of points. For example one could omit point $\vecx_5$ from the list in Figure~\ref{crossings} without changing the curve.
Another consequence of Proposition~\ref{conn-sect} is that if $S$ and $\hat S$ generate the same curve $\gamma$ then $\vf(S) = \vf({\hat S})$.

In general if one deletes a point from a list, this may decrease the variation factor, but it will never increase it. The following proposition appeared in \cite{DL2}.

\begin{proposition}\label{add-point}
Let $S$ be an ordered list points, and let $S^+$ be a list
formed by adding an additional element into the list at some point. Then for any
line $\ell$ in the plane $\vf(S,\ell) \le \vf(S^+, \ell)$ and hence $\vf(S) \le \vf(S^+)$.
\end{proposition}

It remains now to verify that the original definition of $\var(f,\sigma)$ agrees with the one used since \cite{DL1}, that is, that
  \begin{equation}\label{varf-defs}
   \sup_{\gamma \in \Gamma_L} \frac{\cvar(f,\gamma)}{\vf(\gamma)}
   = \sup_S \frac{\cvar(f,S)}{\vf(S)}.
   \end{equation}
Here $\cvar(f,\gamma) = \sup \sum_{j=1}^m |f(\gamma(t_j)) - f(\gamma(t_{j-1}))|$ where the supremum is taken over all partitions $0 \le t_0 < t_1 < \dots < t_m \le 1$ such that $\gamma(t_j) \in \sigma$ for all $j$. If $\gamma$ does not intersect $\sigma$ or if it only meets $\sigma$ at a single point, then we set $\cvar(f,\gamma) = 0$.
The right-hand side is as defined in Section~\ref{S:Prelim}.

Given an ordered list $S \subseteq \sigma$, one may form $\gamma_S \in \Gamma_L$. It is clear that $\cvar(f,S) \le \cvar(f,\gamma_S)$ and $\vf(S) = \vf(\gamma_S)$  so the left-hand side of (\ref{varf-defs}) is at least as large as the right-hand side.

On the other hand, suppose that $\gamma \in \Gamma_L$ and that $[t_j]_{j=0}^m$ is an ordered list of values such that $\vecv_j = \gamma(t_j)$ lies in $\sigma$ for all $j$. Let ${\hat S} = [\vecv_0,\vecv_1, \dots,\vecv_m]$. One can write $\gamma$ as $\gamma_S$ for an ordered list of points $S$ which contains the points $\vecv_j$ in the appropriate order (see Figure~\ref{hat-S}). Since ${\hat S}$ is a sublist of $S$, by Proposition~\ref{add-point}, $\vf({\hat S}) \le \vf(S)$ and so
  \[ \frac{\sum_{j=1}^m |f(\gamma(t_j)) - f(\gamma(t_{j-1}))|}{\vf(\gamma)}
  = \frac{\sum_{j=1}^m |f(\gamma(t_j)) - f(\gamma(t_{j-1}))|}{\vf(S)}
  \le \frac{\cvar(f,{\hat S})}{\vf({\hat S})}.
  \]
This provides the reverse inequality and hence the two sides of (\ref{varf-defs}) are equal.

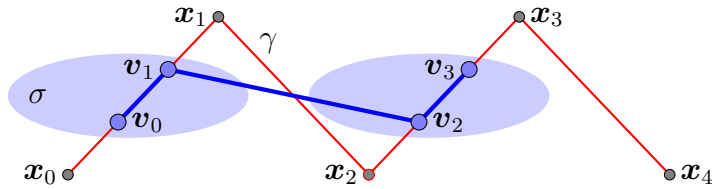
\begin{figure}[!ht]
\begin{center}
\begin{tikzpicture}[xscale=2.0,yscale=0.7]

\fill[blue!20] (0.4,1.5) circle (0.8cm);
\fill[blue!20] (2.4,1.5) circle (0.8cm);
\draw (-0.2,1.5) node {$\sigma$};

 \draw[thick, red] (0,0) -- (1,3) -- (2,0) -- (3,3) -- (4,0);
\draw (0,0) node[left] {$\vecx_0$};
\draw (0,0) node[circle, draw, fill=black!50,inner sep=0pt, minimum width=4pt] {};
\draw (1,3) node[left] {$\vecx_1$};
\draw (1,3) node[circle, draw, fill=black!50,inner sep=0pt, minimum width=4pt] {};
\draw (2,0) node[left] {$\vecx_2$};
\draw (2,0) node[circle, draw, fill=black!50,inner sep=0pt, minimum width=4pt] {};
\draw (3,3) node[right] {$\vecx_3$};
\draw (3,3) node[circle, draw, fill=black!50,inner sep=0pt, minimum width=4pt] {};
\draw (4,0) node[right] {$\vecx_4$};
\draw (4,0) node[circle, draw, fill=black!50,inner sep=0pt, minimum width=4pt] {};

\draw[ultra thick, blue] (0.3333333,1) -- (0.6666666,2) -- (2.333333,1) -- (2.666666,2);
\draw (0.3333333,1) node[circle, draw, fill=blue!50,inner sep=0pt, minimum width=6pt] {};
\draw (0.35,1) node[right] {$\vecv_0$};
\draw (0.6666666,2) node[circle, draw, fill=blue!50,inner sep=0pt, minimum width=6pt] {};
\draw (0.65,2) node[left] {$\vecv_1$};
\draw (2.333333,1) node[circle, draw, fill=blue!50,inner sep=0pt, minimum width=6pt] {};
\draw (2.35,1) node[right] {$\vecv_2$};
\draw (2.666666,2) node[circle, draw, fill=blue!50,inner sep=0pt, minimum width=6pt] {};
\draw (2.65,2) node[left] {$\vecv_3$};

\draw (1.2,2.5) node[right] {$\gamma$};

\draw[red] (2,0) node[circle, draw, fill=black!50,inner sep=0pt, minimum width=4pt] {};
\end{tikzpicture}
\caption{$\gamma = \gamma_{S_0}$ for
$S_0 = [\vecx_0,\vecx_1,\vecx_2,\vecx_3,\vecx_4]$.
If $\vecv_j = \gamma(t_j)$, are the chosen points on $\gamma$ which lie in $\sigma$, we can form
$S = [\vecx_0,\vecv_0,\vecv_1,\vecx_1,\vecx_2,\vecv_2,\vecv_3,\vecx_3,\vecx_4]$ and $\gamma = \gamma_S$ too. Then ${\hat S} = [\vecv_0,\vecv_1,\vecv_2,\vecv_3]$ is a sublist of $S$ so $\vf({\hat S}) \le \vf(S) = \vf(\gamma)$.}\label{hat-S}
\end{center}
\end{figure}

\begin{acknowledgements}
The work of the first author was financially supported
by the Ministry of Higher Education and Scientific Research
of Iraq.
The authors would also like to thank the referee for several suggestions which have improved the paper, and Alan Stoneham for his assistance.
\end{acknowledgements}

\end{document}